\documentclass[11pt]{amsart}

\usepackage{enumerate}

\usepackage[margin=1.5in]{geometry}                

\usepackage{graphicx}
\usepackage{amssymb}
\usepackage{float}
\usepackage{hyperref} 
\usepackage{bbm}
\usepackage{amsmath,amscd}
\usepackage{algorithm,algorithmic}
\usepackage{latexsym}
\usepackage{graphics}
\usepackage{color}
\usepackage{cite}
\usepackage{mathrsfs}
\usepackage{caption}
\usepackage{subcaption}
\usepackage{tikz}
\usetikzlibrary{automata, positioning}
\usetikzlibrary{matrix}
\usepackage{pgfplots}

\newtheorem{theorem}{Theorem}[section]
\newtheorem{lemma}[theorem]{Lemma}
\newtheorem{claim}[theorem]{Claim}

\newtheorem{corollary}[theorem]{Corollary}

\theoremstyle{definition}
\newtheorem{definition}[theorem]{Definition}
\newtheorem{example}[theorem]{Example}

\theoremstyle{remark}
\newtheorem{remark}[theorem]{Remark}

\numberwithin{equation}{section}

\usepackage[colorinlistoftodos,prependcaption,textsize=tiny]{todonotes}

\newcommand{\cH}{\mathcal{H}}

\newcommand{\cS}{\mathcal{S}}

\newcommand{\cU}{\mathcal{U}}
\newcommand{\cN}{\mathcal{N}}

\newcommand{\cP}{\mathcal{P}}

\newcommand{\bP}{\mathbb{P}}

\newcommand{\bE}{\mathbb{E}}

\newcommand{\R}{\mathbb{R}}

\renewcommand{\S}{\mathcal{S}}

\newcommand{\ep}{\epsilon }

\newcommand{\de}{\delta }

\newcommand{\si}{\sigma }

\newcommand{\ga}{\gamma }
\newcommand{\Ga}{\Gamma }
\newcommand{\be}{\beta }

\newcommand{\one}{\mathbf{1}}

\newcommand{\supp}{\operatorname{supp}}

\newcommand{\cReff}{\eff\mathcal{R}}

\newcommand{\eff}{\textrm{eff}}

\newcommand{\UST}{\operatorname{UST}}

\newcommand{\detp}{\det{}'}

\newcommand{\KL}{\text{KL}}

\newcommand{\rv}[1]{\underline{#1}}

\newcommand{\bi}{\begin{itemize}}
\newcommand{\ei}{\end{itemize}}

\newcommand{\MEO}{\operatorname{MEO}}

\pgfplotsset{compat=1.18}

\usepackage[foot]{amsaddr}

\begin{document}
\large

\title{Minimizing the determinant of the graph Laplacian}

\author[Albin]{Nathan Albin$^1$}
\address{$^1$Department of Mathematics, Kansas State University, Manhattan, KS 66506}
\email{albin@ksu.edu}
\author[Lind]{Joan Lind$^2$}
\address{$^2$Department of Mathematics, University of Tennessee, Knoxville, TN 37996, USA}
\email{jlind@utk.edu}
\author[Melikyan]{Anna Melikyan}
\email{meliqyananna@gmail.com}
\author[Poggi-Corradini]{Pietro Poggi-Corradini$^1$}
\email{pietro@ksu.edu}

\thanks{The first and last authors acknowledge support from NSF grant  n.~1515810}

\begin{abstract}
In this paper, we study extremal values for the determinant of the weighted graph Laplacian under simple nondegeneracy conditions on the weights.  We derive necessary and sufficient conditions for the determinant of the Laplacian to be bounded away from zero and for the existence of a minimizing set of weights.  These conditions are given both in terms of properties of random spanning trees and in terms of a type of density on graphs. These results generalize and extend the work of \cite{melikyan2019}.
\end{abstract}

\date{}


\maketitle

\section{Introduction}\label{sec:introduction}

\subsection{Determinant of the Laplacian matrix}

Consider a connected, undirected graph $G=(V,E)$, with vertex set $V$ and edge set $E$ that are both finite. Multi-edges are allowed, but not self-loops. To each edge $e\in E$, we assign a positive weight $\sigma(e)>0$.  We shall tend to treat these weights as entries in a vector $\sigma\in\mathbb{R}^E_{>0}$, the set of positive vectors indexed by the edge set.  Since $G$ is assumed connected, the weighted Laplacian matrix $L_\sigma$ has $|V|-1$ positive eigenvalues.  If we order the eigenvalues so that
\begin{equation*}
0 = \lambda_1 < \lambda_2 \le \lambda_3 \le\cdots \le \lambda_{|V|},
\end{equation*}
then the determinant of the Laplacian is defined as
\begin{equation*}
\detp L_\sigma := \prod_{i=2}^{|V|}\lambda_i = \det\left(L_\sigma + \frac{1}{|V|}\one\one^T\right).
\end{equation*}
Kirchhoff's matrix tree theorem provides a way to compute this determinant using weighted spanning trees. Let $\Ga$ be the set of all spanning trees of $G$. For $\gamma\in\Gamma$, we define the $\sigma$-weight of $\gamma$ as
\begin{equation*}
\sigma[\gamma] = \prod_{e\in\gamma}\sigma(e) = \prod_{e\in E}\sigma(e)^{\mathcal{N}(\gamma,e)},
\end{equation*}
where $\mathcal{N}$ is the $\Gamma\times E$ \emph{edge usage matrix} defined as
\begin{equation}\label{eq:usage-matrix}
\mathcal{N}(\gamma,e) = 
\begin{cases}
1 & \text{if }e\in\gamma,\\
0 & \text{if }e\notin\gamma.
\end{cases}
\end{equation}
(Since we will sometimes use additive weights on edges, we use the square braces in the notation $\sigma[\gamma]$ to help call attention to the fact that $\sigma$ is used as a \emph{multiplicative} weight in this definition.)
\begin{theorem}[Kirchhoff's Matrix Tree Theorem]\label{thm:matrix-tree}
For any positive set of edge weights $\sigma\in\mathbb{R}^E_{>0}$,
\begin{equation*}
\detp L_\sigma = |V|\sum_{\gamma\in\Gamma}\sigma[\gamma].
\end{equation*}
\end{theorem}

Since the determinant is positively homogeneous of degree $|V|-1$,  one way to gain a deeper understanding of $\detp L_\sigma$, is to consider its extreme values subject to some normalizing assumption on $\sigma$. In particular, we consider the following optimization problem, which for purpose of this paper, we call the {\it Minimum Determinant Problem}.
\begin{equation}\label{eq:min-det-orig}
  \begin{split}
  \underset{\sigma>0}{\text{minimize}}\quad&\sum_{\gamma\in\Gamma}\sigma[\gamma]\\
  \text{subject to}\quad&\prod_{e\in E}\sigma(e) = 1.
  \end{split}
\end{equation}

If $G$ is not a tree, then the determinant can be made arbitrarily large with weights satisfying this constraint.  To see this, let $e$ be any edge and define $\sigma$ to be $M>0$ on $e$ and $M^{-1/(|E|-1)}$ on all remaining edges, in order to satisfy the product constraint.  If $\gamma$ is any spanning tree containing $e$, then
\begin{equation*}
\sigma[\gamma] = M\cdot M^{-\frac{|V|-2}{|E|-1}} = M^{\frac{|E|-(|V|-1)}{|E|-1}}.
\end{equation*}
Assuming that $G$ is not a tree, the exponent on $M$ is positive and, therefore, $\sigma[\gamma]$ can be made arbitrarily large by sending $M$ to $+\infty$. Hence, we will focus on the minimum determinant problem.

Our main result in this paper is that, as the examples below suggest, the optimization problem in~\eqref{eq:min-det-orig} divides the set of  all graphs into three categories,
\begin{itemize}
  \item graphs, like the paw, for which the determinant can be made arbitrarily close to zero,
  \item graphs, like the diamond, for which the determinant is bounded away from zero and the minimum value is attained by some (unique) choice of weights, and
  \item graphs, like the house, for which the determinant is bounded away from zero, but no minimizing choices of weights exists.
\end{itemize}
Furthermore, as we will show in Theorem \ref{thm:mincore}, graphs like the paw always have a subgraph, we call {\it minimal core}, that behaves like the diamond graph, and moreover, the minimal core can be shrunk and a new minimal core can be found for the shrunk graph in an iterative manner. We call this a {\it deflation process}. This result will be used in \cite{alpc}, where we explore the scaling limit of fair trees on planar square grids.

\subsubsection{Examples}\label{ssec:examples}
The following three examples will be useful in demonstrating the results of this paper.

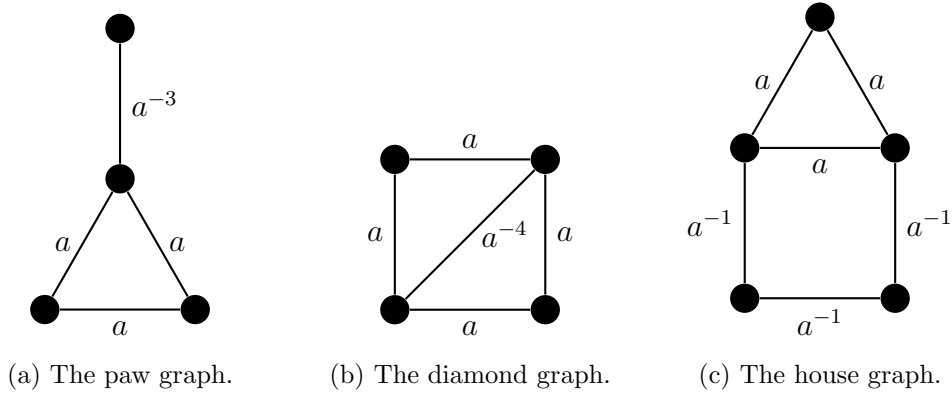
\begin{figure}
\centering
\begin{subfigure}[b]{0.333\textwidth}
\centering
\begin{tikzpicture}[scale=2]
\tikzstyle{node}=[fill=black,circle]
\tikzstyle{edge}=[draw=black,thick]
\node[node] (a) at (0,0) {};
\node[node] (b) at (0,-1) {};
\node[node] (c) at (-0.5,-1.87) {};
\node[node] (d) at (0.5,-1.87) {};

\draw[edge] (a)--(b) node [midway, right] {$a^{-3}$};
\draw[edge] (b)--(c) node [midway, left] {$a$};
\draw[edge] (b)--(d) node [midway, right] {$a$};
\draw[edge] (c)--(d) node [midway, below] {$a$};
\end{tikzpicture}
\caption{The paw graph.}\label{fig:paw}
\end{subfigure}%
\begin{subfigure}[b]{0.333\textwidth}
\centering
\begin{tikzpicture}[scale=2]
\tikzstyle{node}=[fill=black,circle]
\tikzstyle{edge}=[draw=black,thick]
\node[node] (a) at (0,0) {};
\node[node] (b) at (1,0) {};
\node[node] (c) at (1,1) {};
\node[node] (d) at (0,1) {};

\draw[edge] (a)--(b) node [midway, below] {$a$};
\draw[edge] (b)--(c) node [midway, right] {$a$};
\draw[edge] (c)--(d) node [midway, above] {$a$};
\draw[edge] (d)--(a) node [midway, left] {$a$};
\draw[edge] (a)--(c) node [midway, right] {$a^{-4}$};
\end{tikzpicture}
\caption{The diamond graph.}\label{fig:diamond}
\end{subfigure}%
\begin{subfigure}[b]{0.333\textwidth}
\centering
\begin{tikzpicture}[scale=2]
\tikzstyle{node}=[fill=black,circle]
\tikzstyle{edge}=[draw=black,thick]
\node[node] (b) at (0,-1) {};
\node[node] (c) at (-0.5,-1.87) {};
\node[node] (d) at (0.5,-1.87) {};
\node[node] (e) at (-0.5,-2.87) {};
\node[node] (f) at (0.5,-2.87) {};

\draw[edge] (b)--(c) node [midway, left] {$a$};
\draw[edge] (b)--(d) node [midway, right] {$a$};
\draw[edge] (c)--(d) node [midway, below] {$a$};
\draw[edge] (c)--(e) node [midway, left] {$a^{-1}$};
\draw[edge] (d)--(f) node [midway, right] {$a^{-1}$};
\draw[edge] (e)--(f) node [midway, below] {$a^{-1}$};
\end{tikzpicture}
\caption{The house graph.}\label{fig:house}
\end{subfigure}
\caption{Graphs for Examples~\ref{ex:det-paw}--~\ref{ex:det-house}.}
\end{figure}

\begin{example}\label{ex:det-paw}
  Let $G$ be the paw graph in Figure~\ref{fig:paw}. Then the infimum in~\eqref{eq:min-det-orig} is zero (and no minimizing choice of $\sigma$ exists).
  To see this, consider the edge weights $\sigma$ as given by the labels in the figure.  There are three spanning trees, each using the pendant edge and two of the remaining three edges, so
\begin{equation*}
\sum_{\gamma\in\Gamma}\sigma[\gamma] = 3a^{-1}.
\end{equation*}
By making $a$ arbitrarily large, we can make $\detp L_\sigma$ arbitrarily close to $0$. Since $\detp L_\sigma>0$ for any choice of positive weights, no minimizing choice of $\sigma$ exists.
\end{example}

\begin{example}\label{ex:det-C3}
Consider the diamond graph, as in Figure~\ref{fig:diamond}, with positive edge weights $\si$ that are functions of $a>0$, as depicted. Note that there are $8$ total spanning trees, $4$ that contain the diagonal and $4$ that do not. So,
\begin{equation}\label{eq:diamond-detprod}
\sum_{\gamma\in\Gamma}\sigma[\gamma] = 4a^3 + 4a^{-2}.
\end{equation}
The quantity on the right is minimized when $a\equiv\left(\frac{2}{3}\right)^{\frac{1}{5}}$. We claim that this choice of weights $\bar{\si}$ is optimal for problem (\ref{eq:min-det-orig}). To see this, we first compute the induced  (weighted uniform) pmf $\mu_{\bar{\si}}$ on spanning trees, see Section \ref{ssec:wust}, where $\mu_{\bar{\si}}(\ga)$ is proportional to $\bar{\sigma}[\gamma]$. The normalizing factor (\ref{eq:diamond-detprod}) becomes
\begin{equation}\label{eq:diamong-detprod-computed}
\sum_{\gamma\in\Gamma}\bar{\sigma}[\gamma]=4\left(\frac{2}{3}\right)^{\frac{3}{5}}+4\left(\frac{2}{3}\right)^{-\frac{2}{5}}=4\left(\frac{2}{3}\right)^{\frac{3}{5}}\left(1+\frac{3}{2}\right)=10\left(\frac{2}{3}\right)^{\frac{3}{5}}.
\end{equation}
So, if $\ga\in\Ga$ does not include the diagonal, then
\[
\mu_{\bar{\si}}(\ga)=\frac{a^3}{4a^3+4a^{-2}}=\frac{\left(\frac{2}{3}\right)^{\frac{3}{5}}}{10\left(\frac{2}{3}\right)^{\frac{3}{5}}}=\frac{1}{10}.
\]
While, if $\ga$ does contain the diagonal edge, then
\[
\mu_{\bar{\si}}(\ga)=\frac{a^{-2}}{4a^3+4a^{-2}}=\frac{3}{20}.
\]
Next, we compute the edge probabilities of $\mu_{\bar{\si}}$ which, for $e\in E$, we denote as $\eta_{\bar{\si}}(e):=\bP_{\mu_{\bar{\si}}}(e\in\underline{\ga})$. If $e$ is not the diagonal, then
\[
\eta_{\bar{\si}}(e)=3\cdot\frac{1}{10}+2\cdot\frac{3}{20}=\frac{3}{5}.
\]
While, if $e_0$ is the diagonal edge, then
\[
\eta_{\bar{\si}}(e_0)= 4\cdot\frac{3}{20}=\frac{3}{5}.
\]
Note that $\frac{3}{5}$ is equal to $\frac{|V|-1}{|E|},$ and this is actually necessary whenever the edge-probabilities are constant, because they always add up to $|V|-1$, see (\ref{eq:edge-probabilities-sum}).
As a result, according to Definition \ref{def:homogeneous}, the diamond graph $G$ is strictly homogeneous (with respect to constant $\beta$).
Consequently, Theorem \ref{thm:main} implies that problem (\ref{eq:min-det-orig}) for $G$ has a unique optimal set of weights $\si^*$.

We claim that our choice of weights $\bar{\si}$ is the optimal set of weights $\si^*$. To see this, we need to appeal to convex duality. As we will describe in Section \ref{sec:dual-entropy}, if $p^*$ is the minimum determinant, namely, the value of (\ref{eq:min-det-orig}), then by Lemma \ref{lem:entropy-duality}, $p^*=\exp(H^*)$, where $H^*$ is the value of the following maximum entropy problem
\begin{equation}\label{eq:max-entropy-u}
\underset{\mu\in\mathcal{P}(\Gamma)\cap\mathcal{U}}{\text{maximize}}\quad H(\mu),
\end{equation}
where $\cP(\Ga)$ is the set of all pmfs on $\Ga$, see (\ref{eq:set-all-pmfs}), the Shannon entropy of a pmf $\mu$, see \cite{shannon:bell1948}, is
\begin{equation}\label{eq:shannon-entropy}
H(\mu) := -\sum_{\gamma\in\Gamma}\mu(\gamma)\log\mu(\gamma),
\end{equation}
(we interpret $\mu(\gamma)\log \mu(\gamma)=0$ if $\mu(\gamma)=0$), and $\cU$ is the convex cone
\begin{equation}\label{eq:cone-u}
\mathcal{U} := \left\{u\in\mathbb{R}^\Gamma_{\ge 0} : \sum_{\gamma\in\Gamma}\mathcal{N}(\gamma,e) u(\gamma) = \frac{|V|-1}{|E|}u(\Gamma),\ \text{ for all }e\in E \right\}.
\end{equation}
To conclude the example then, we first note that (\ref{eq:diamong-detprod-computed}) gives us the upper bound
\[
p^*\le 10\left(\frac{2}{3}\right)^{\frac{3}{5}}.
\]
On the other hand, since the edge-probabilities of $\mu_{\bar{\si}}$ are constant on $E$, we have $\mu_{\bar{\si}}\in\cU$. Hence, computing the entropy we find that
\[
H(\mu_{\bar{\si}})=-\left(4\cdot\frac{1}{10}\log\frac{1}{10}+4\cdot\frac{3}{20}\log\frac{3}{20}\right)=\log\left(10\left(\frac{2}{3}\right)^{\frac{3}{5}}\right)
\]
So by duality $\exp(H(\mu_{\bar{\si}}))=10\left(\frac{2}{3}\right)^{\frac{3}{5}}$ is also a lower bound for $p^*$. This shows that our choice of weights $\bar{\si}$ was optimal.
\end{example}

\begin{example}\label{ex:det-house}
Consider the house graph with positive edge weights $\si$ assigned as in Figure~\ref{fig:house}. There are $11$ spanning trees of the house graph.
Let $T$ be the triangle where the edge weight is $\si\equiv a$ and let $S$ be the edge-set where $\si\equiv a^{-1}$. A spanning tree for $G$ can be created by removing an edge from $T$ ($3$ choices) and one from $S$ ($3$ choices). This gives rise to $9$ trees that have $\sigma$-weight equal to $1$. The other two trees are obtained by removing two edges from $T$ and none from $S$ and thus have $\sigma$-weight equal to $a^{-2}$.  Therefore,
\begin{equation}\label{eq:house-limit}
\lim_{a\to\infty}\sum_{\gamma\in\Gamma}\sigma[\gamma] = 9.
\end{equation}
It turns out that this is a minimizing sequence for problem (\ref{eq:min-det-orig}) on $G$. However, in this case, there is no optimal set of edge-weights, by Theorem \ref{thm:main}, because
the house graph is homogeneous, but is not strictly-homogeneous. To verify homogeneity, let $\mu_\si$ be the pmf induced by the weights $\si$ and consider the pmf $\mu^*:=\lim_{a\to\infty}\mu_\si$, that is uniform on the $9$ trees that remove an edge from $T$ and one from $S$. Then, $\mu^*$ induces a constant edge usage probability, so $G$ is homogeneous.
On the other hand, no pmf $\mu'$ that is induced from a set of edge-weights can have this property. To see why, note that we would necessarily have $\mu'(\gamma)>0$, for all $\gamma\in\Gamma$. However, nine of the spanning trees intersect $S$ on two edges, while the remaining two intersect $S$ on all three. Therefore, if we average the edge-probabilities $\eta'$ of $\mu'$ over the set $S$, we obtain
\begin{align*}
\frac{1}{|S|}\sum_{e\in S}\eta'(e) & =
  \frac{1}{3}\sum_{e\in S}\sum_{\gamma\in\Gamma}\mathcal{N}(\gamma,e)\mu'(\gamma) \\
&= \frac{1}{3}\sum_{\gamma\in\Gamma}\mu'(\gamma)\sum_{e\in S}\mathcal{N}(\gamma,e) > \frac{2}{3} = \frac{|V|-1}{|E|},
\end{align*}
showing that we cannot have $\eta'(e)\equiv\frac{|V|-1}{|E|}.$ 

Finally, we claim that $p^*=9$, which will imply that $\mu_\si$ is a minimizing sequence for problem (\ref{eq:min-det-orig}). On one hand, by (\ref{eq:house-limit}), we have $p^*\le 9.$ On the other hand, if we compute the entropy
\begin{align*}
H(\mu_\si) & =-\left(9\cdot \frac{1}{9+2a^{-2}}\log\frac{1}{9+2a^{-2}}+2\cdot\frac{a^{-2}}{9+2a^{-2}}\log\frac{a^{-2}}{9+2a^{-2}}\right),
\end{align*}
we see that $\exp(H(\mu_\si))$ converges to $9$ as $a\rightarrow\infty$.

Therefore, we have shown that this construction provides a minimizing sequence of the determinant of the Laplacian, but that no minimizer exists.
\end{example}

\subsubsection{Reweighting}
As will become evident in the sequel, the classification just described is highly imbalanced; very few graphs have the property that a minimizing choice of $\sigma$ exists. Indeed, most graphs are like the paw, in that the determinant can be made arbitrarily close to zero. One can locate a set of edges (like the pendulum in the paw) on which sending $\sigma$ to zero cannot be adequately compensated for by sending $\sigma$ on other edges to $\infty$. In a sense, certain edges are too ``powerful'' in the determinant.

In order to address this power imbalance, it is useful to consider a slightly more general form of normalizing assumption.  In this paper, we study the following parameterized minimum determinant problem
\begin{equation}\label{eq:min-det}
\begin{split}
\underset{\sigma>0}{\text{minimize}}\quad&\sum_{\gamma\in\Gamma}\sigma[\gamma]\\
\text{subject to}\quad&\prod_{e\in E}\sigma(e)^{\beta(e)} = 1,
\end{split}
\end{equation}
where the weights $\beta(e)$ satisfying the conditions
\begin{equation}\label{eq:beta-conditions}
\beta(e) > 0 \quad\text{for all }e\in E\quad\text{and}\quad \beta(E) := \sum_{e\in E}\beta(e) = 1.
\end{equation}

\begin{remark}\label{rem:rescale-beta}
The requirement that the $\beta(E)=1$ is not restrictive. If $\tilde{\beta}$ is a set of positive weights, then replacing $\tilde{\beta}$ by $\beta:=\tilde{\beta}/\tilde{\beta}(E)$ does not affect the constraint equation in~\eqref{eq:min-det}. Normalizing $\beta$ in this way is convenient in that it allows us to omit the term $\beta(E)$ from many formulas that follow.  Moreover, it provides $\beta$ with the interpretation of a probability distribution on the edges.
\end{remark}

To see how reweighting helps, consider the paw graph again, but with $\beta$ equal to $1/3$ on the pendant edge and $2/9$ on the remaining three edges. If $\sigma\equiv a$ on the three triangle edges, then the constraint requires that $\sigma=a^{-2}$ on the pendant edge. Thus,
\begin{equation*}
\sum_{\gamma\in\Gamma} \sigma[\gamma] = 3,
\end{equation*}
regardless the choice of $a$ and this can be shown to be optimal. In fact, Theorem~\ref{thm:good-beta} will show that for any graph there are choices of weights $\beta$ that allow the determinant to be minimized and, therefore, to be bounded away from zero.

\subsubsection{Goal}
The goal of this paper is to explain the behavior of~\eqref{eq:min-det} in terms of the weights $\beta$ and two structural graph properties. The first of these structural properties is related to random spanning trees on the graph; the second is related to a concept of combinatorial density.

\subsection{Random spanning trees and random edges}\label{sec:random-trees-edges}
For any finite set $X$, the set of all probability mass functions (pmfs) on $X$ and the set of all positive pmfs on $X$ are denoted respectively by
\begin{equation}\label{eq:set-all-pmfs}
 \cP(X) := \left\{\mu\in\mathbb{R}^{X}_{\ge 0} : \sum_{x\in X}\mu(x)=1\right\}
 \end{equation}
 and
 \begin{equation*}
 \cP_+(X) := \left\{\mu\in\mathbb{R}^{X}_{> 0} : \sum_{x\in X}\mu(x)=1\right\}.
\end{equation*}
Using this notation,~\eqref{eq:beta-conditions} is simply the requirement that $\beta$ belongs to $\mathcal{P}_+(E)$.

A pmf $\mu\in\cP(X)$ defines a random element $\underline{x}\in X$ that satisfies the law
\[
\bP_\mu\left(\underline{x}=x\right)=\mu(x)\qquad\forall x\in X.
\]
In this paper, we shall focus on two particular choices for $X$, namely $\Gamma$ and $E$.  When $X=\Gamma$, we will call the corresponding random object $\rv{\gamma}\in\Gamma$ a \emph{random spanning tree}, and when $X=E$, we will call the corresponding random object $\rv{e}\in E$ a \emph{random edge}.

There is an interesting connection between random spanning trees and random edges.  First, we observe that each pmf $\mu\in\cP(\Gamma)$ induces a marginal \emph{edge usage probability} on each edge $e\in E$, giving the probability that $e$ is in the random tree $\rv{\gamma}$. This probability can be computed as
\begin{equation*}
\eta_\mu(e) := \mathbb{P}_\mu(e\in\rv{\gamma}) = \sum_{\gamma\in\Gamma}\mathcal{N}(\gamma,e)\mu(\gamma).
\end{equation*}
Then, $\eta_\mu\in\mathbb{R}^E_{\ge 0}$ is called the \emph{edge usage probability vector} associated with $\mu$.

Next, we note that each pmf, $\mu\in\cP(\Gamma)$, on the set of spanning trees can be made to induce a pmf on the set of edges using the following procedure.  To select a random edge $\rv{e}$, first select a random spanning tree $\rv{\gamma}$ according to $\mu$.  This tree will have $|V|-1$ edges.  Choose $\rv{e}$ uniformly at random from among them.  The pmf $\beta\in\cP(E)$ associated with this random edge can be computed as follows.
\begin{equation}\label{eq:beta-eta-verified}
\beta(e) = \frac{1}{|V|-1}\sum_{\gamma\in\Gamma}\mathcal{N}(\gamma,e)\mu(\gamma) = \frac{\eta_\mu(e)}{|V|-1}.
\end{equation}

\subsection{Weighted uniform spanning trees}\label{ssec:wust}
The {\it $\si$-weighted uniform spanning trees} ($\UST_\si$) are a special class of random spanning trees obtained by fixing edge weights $\si\in\R_{>0}^E$. These are random trees $\underline{\ga}\in\Ga$ whose probability density $\mu_\si(\ga)$ is proportional to $\sigma[\gamma]$. It is well known, see \cite[Chapter 4]{LP:book}, that a $\UST_\si$ can be sampled using weighted random walks, as in the Aldous-Broder algorithm, or using loop-erased weighted random walks, as in Wilson's algorithm. Moreover, another famous result of Kirchhoff states that the edge usage probabilities for a $\UST_\si$ are related to the effective resistances $\cReff_\si$ of the associated electrical network with edge conductances given by $\si$. More precisely, if $e\in E$, then
\begin{equation}\label{eq:peredgeeffres}
\eta_\si(e) := \bP_{\mu_\si}\left(e\in \underline{\ga}\right) = \si(e)\cReff_\si(e).
\end{equation}

\subsection{Homogeneous graphs}

We now introduce a generalization of the concept of homogeneous graph in~\cite{achpcst:disc-math2021} inspired by~\eqref{eq:beta-eta-verified}.

\begin{definition}\label{def:homogeneous}
Given a pmf $\beta\in\cP(E)$, the graph $G$ is called \emph{$\beta$-homogeneous} if there exists a pmf $\mu\in\mathcal{P}(\Gamma)$ whose edge usage probability vector $\eta_\mu$ satisfies
\begin{equation}\label{eq:beta-eta}
\frac{\eta_\mu(e)}{|V|-1} = \beta(e).
\end{equation}
$G$ is called \emph{strictly $\beta$-homogeneous} if there exists a $\UST_\si$ pmf $\mu_\sigma$ satisfying this condition.
\end{definition}

In other words, a graph is (resp. strictly)  $\beta$-homogeneous if and only if there is a (resp. $\UST_\si$) pmf on the spanning trees that induces the edge pmf $\beta$ through the procedure described in Section~\ref{sec:random-trees-edges}.

\begin{remark}
In the case that $\beta$ is the constant vector, the concept of $\beta$-homogenous coincides with the concept of \emph{homogeneous} in~\cite{achpcst:disc-math2021}. We shall drop the $\beta$ from the terminology in this case.
\end{remark}

One of the main results of this paper is the following theorem, which is proved in Section~\ref{sec:barriers}.
\begin{theorem}\label{thm:main}
The determinant in~\eqref{eq:min-det} is bounded away from zero if and only if $G$ is $\beta$-homogeneous.  A minimizing choice of $\sigma$ exists if and only if $G$ is strictly $\beta$-homogeneous.
\end{theorem}

\subsection{Dense graphs}

An alternative characterization can also be stated in terms of a type of density measure on graphs.  In order to characterize a graph as \emph{dense}, we will compare its density to the density of its subgraphs.  To this end, we define a family of subgraphs.
\begin{definition}\label{def:subgraphs}
Let $H=(V_H,E_H)$ be a subgraph of $G=(V,E)$.  Then, $H$ is called \emph{edge-induced} if it can be obtained by removing all edges in $E\setminus E_H$ from $G$ and deleting any vertices with degree zero. Also, $H$ is called \emph{nontrivial} if $|E_H|\ge 1$.  In the following, we shall use $\S(G)$ to denote the family of all connected, nontrivial, edge-induced subgraphs of $G$.
\end{definition}

\begin{definition}
For $\beta\in\cP(E)$ and $H\in\S(G)$, define the \emph{$\beta$-density} of $H$ as
\begin{equation*}
\theta_\beta(H) := \frac{\beta(E_H)}{|V_H|-1},\quad\text{where}\quad
\beta(E_H)=\sum_{e\in E_H}\beta(e).
\end{equation*}
\end{definition}

\begin{remark}
The concept of $\beta$-density here is proportional to the $1$-density in~\cite{achpcst:disc-math2021} when $\beta$ is constant;
\begin{equation*}
  \theta_\beta(H) = \frac{|E_H|/|E|}{|V_H|-1} = \frac{1}{|E|}\cdot\frac{|E_H|}{|V_H|-1} = \frac{1}{|E|}\theta(H),
\end{equation*}
where $\theta(H)$ is the \emph{$1$-density} defined in~\cite{achpcst:disc-math2021}.
\end{remark}

\begin{definition}\label{def:beta-dense}
A graph $G$ is called \emph{$\beta$-dense} if $\theta_\beta(H) \le \theta_\beta(G)$ for all $H\in\S(G)$. On the other hand, $G$ is called \emph{strictly $\beta$-dense} if the inequality is strict whenever $H\ne G$.
\end{definition}

In Section~\ref{sec:homog-and-density} we prove the following two lemmas.

\begin{lemma}\label{lem:dense-is-homogeneous}
A graph is $\beta$-dense if and only if it is $\beta$-homogeneous.
\end{lemma}

\begin{lemma}\label{lem:sdense-is-shomogeneous}
A graph is strictly $\beta$-dense if and only if it is strictly $\beta$-homogeneous and biconnected.
\end{lemma}

This provides an alternative to Theorem~\ref{thm:main}, which is also proved in Section~\ref{sec:homog-and-density}.
\begin{theorem}\label{thm:main-dense}
The determinant in~\eqref{eq:min-det} is bounded away from zero if and only if $G$ is $\beta$-dense.  A minimizing choice of $\sigma$ exists if and only if $G$ is $\beta$-dense and every biconnected component of $G$ is strictly $\beta$-dense.
\end{theorem}

Once again, the three example graphs in Section \ref{ssec:examples} can be instructive. The paw graph is not $1$-dense (i.e., not $\beta$-dense with constant $\beta$); if $H$ is the subgraph induced by removing the pendant edge, then
\begin{equation*}
  \theta(H) = \frac{3}{2} > \frac{4}{3} = \theta(G).
\end{equation*}
By an exhaustive search, the reader can verify that the diamond graph is strictly $1$-dense. The house graph, on the other hand, is $1$-dense, but not strictly so; the subgraph comprising the triangle at the top of the house has the same $1$-density as the entire graph.

After this work was completed, we learned that minimum determinant problems like (\ref{eq:min-det}) have been studied in other contexts. For instance, we believe that \cite[Theorem 2.10]{anari-gharan-vinzant:duke2021} is related to the optimal weights $\si^*$ we find when $G$ is strictly homogeneous. 

Finally, the results in this paper can be extended to the generality of matroids, see \cite{huytruong:thesis}.

The remainder of this paper is structured as follows.  In Section~\ref{sec:minimizers}, we derive necessary and sufficient conditions for the existence of a minimizer for the determinant.  In Section~\ref{sec:dual-entropy}, we derive a Lagrangian dual problem to the minimum determinant problem; this dual takes the form of an entropy maximization problem.  In Section~\ref{sec:minimizing-sequences}, we study minimizing sequences for the determinant and derive necessary and sufficient conditions for the determinant to be bounded away from zero.  In Section~\ref{sec:barriers}, we summarize the barriers to the existence of a minimizer of the determinant and prove Theorem~\ref{thm:main}. In Section \ref{sec:homog-and-density}, we describe the connection of homogeneity with denseness of subgraphs, and prove Theorem \ref{thm:main-dense}.
In Section \ref{sec:det-bounds}, we explore general determinant lower bounds. Finally, in Section \ref{sec:connections-to-sptmod}, we apply our results to spanning tree modulus.

\section{Minimizers of the determinant}\label{sec:minimizers}

As remarked in Section~\ref{sec:introduction}, the determinant is positively homogeneous of degree $|V|-1$.  The constraint function in~\eqref{eq:min-det}, on the other hand, is positively homogeneous of degree $1$.  Thus, the function
\begin{equation}\label{eq:g-sigma}
g(\sigma) := \log\sum_{\gamma\in\Gamma}\sigma[\gamma] - (|V|-1)\sum_{e\in E}\beta(e)\log\sigma(e)
\end{equation}
is positively homogeneous of degree zero.  Using a standard trick from convex optimization, we can see that any minimizer of~\eqref{eq:min-det} is a minimizer of $g$, and, likewise, any minimizing sequence of~\eqref{eq:min-det} is a minimizing sequence of $g$. This allows a simple characterization of a minimizer.

\begin{lemma}\label{lem:necessary-min}
Suppose $\sigma$ is a minimizer of the determinant in~\eqref{eq:min-det}.  Then $G$ is strictly $\beta$-homogeneous.  In particular, if $\mu_{\sigma}$ is the $\UST_\si$ associated to $\si$, and $\eta_{\sigma}$ is the edge usage probability vector for $\mu_\si$,  then $\eta_\sigma$ and $\beta$ satisfy~\eqref{eq:beta-eta}.
\end{lemma}
\begin{proof}
The gradient of $g$ must vanish at a minimizer.  Note that for any $\gamma\in\Gamma$ and any $e\in E$, we have
\begin{equation*}
\frac{\partial}{\partial\sigma(e)}\sigma[\gamma] = \frac{\partial}{\partial\sigma(e)}\prod_{e'\in E}\sigma(e')^{\mathcal{N}(\gamma,e')}
= \frac{1}{\sigma(e)}\mathcal{N}(\gamma,e)\sigma[\gamma].
\end{equation*}
It follows that, for all edges $e\in E$,
\begin{equation*}
\begin{split}
0 &= \frac{1}{\sigma(e)}\left(\sum_{\gamma\in\Gamma}\sigma[\gamma]\right)^{-1}\sum_{\gamma\in\Gamma}\mathcal{N}(\gamma,e)\sigma[\gamma]
- (|V|-1)\cdot\frac{\beta(e)}{\sigma(e)}\\
&= \frac{1}{\sigma(e)}
\left(
\sum_{\gamma\in\Gamma}\mathcal{N}(\gamma,e)\mu_{\sigma}(\gamma) - (|V|-1)\beta(e)
\right),
\end{split}
\end{equation*}
from which the lemma follows.
\end{proof}

\section{Duality and entropy maximization}\label{sec:dual-entropy}

In this section, we derive a Lagrangian dual problem to the minimum determinant problem~\eqref{eq:min-det}.  We begin with a convenient change of variables.  For any set of positive weights $\sigma$, we define the variables
\begin{equation*}
r(e) = \log\sigma(e)\qquad\text{and}\qquad
\tau(\gamma) = \log\sigma[\gamma].
\end{equation*}
Note that $\tau$ is related to $r$ by the formula
\begin{equation*}
\tau(\gamma) = \log\prod_{e\in E}\sigma(e)^{\mathcal{N}(\gamma,e)} = \sum_{e\in E}\mathcal{N}(\gamma,e)r(e).
\end{equation*}
In these variables,~\eqref{eq:min-det} becomes the convex problem
\begin{equation}\label{eq:min-det-equiv}
\begin{split}
\underset{r\in\mathbb{R}^E,\tau\in\mathbb{R}^\Gamma}{\text{minimize}}\quad&\sum_{\gamma\in\Gamma}\exp\{\tau(\gamma)\}\\
\text{subject to}\quad&\sum_{e\in E}\beta(e)r(e) = 0\\
& \sum_{e\in E}\mathcal{N}(\gamma,e)r(e) \le \tau(\gamma).
\end{split}
\end{equation}
(It is valid to introduce the inequality in the last constraint since the objective function is increasing in each $\tau(\gamma)$.)

Introducing the dual variables $u\in\mathbb{R}^\Gamma_{\ge 0}$ and $t\in\mathbb{R}$, we can write the Lagrangian function as
\begin{align*}
& \mathcal{L}(r,\tau,u,t) \\
&= \sum_{\gamma\in\Gamma}\exp\{\tau(\gamma)\} + 
\sum_{\gamma\in\Gamma}u(\gamma)\left(\sum_{e\in E}\mathcal{N}(\gamma,e)r(e)-\tau(\gamma)\right) - t\sum_{e\in E}\beta(e)r(e)\\
&= \sum_{\gamma\in\Gamma}\exp\{\tau(\gamma)\} - \sum_{\gamma\in\Gamma}u(\gamma)\tau(\gamma)
+ \sum_{e\in E}r(e)\left(\sum_{\gamma\in\Gamma}\mathcal{N}(\gamma,e)u(\gamma) - t\beta(e)\right).
\end{align*}
The dual objective function is defined as
\begin{equation*}
\phi(u,t) := \inf_{r,\tau}\mathcal{L}(r,\tau,u,t).
\end{equation*}

\begin{lemma}\label{lem:slater-applied}
There exists a pair $(u^*,t^*)\in\mathbb{R}^\Gamma_{\ge 0}\times\mathbb{R}$ such that
\begin{align*}
& \phi(u^*,t^*) \\
& =\inf\left\{\sum_{\gamma\in\Gamma}\sigma[\gamma] : \sigma\in\mathbb{R}^E_{>0},\prod_{e}\sigma(e)^{\beta(e)}=1\right\} \\
& = \sup\left\{\phi(u,t):{(u,t)\in\mathbb{R}^\Gamma_{\ge 0}\times\mathbb{R}}\right\}.
\end{align*}
\end{lemma}
\begin{proof}
This is a consequence of Slater's condition (see~\cite[Sec.~5.2.3]{boyd2004convex}).  The primal problem,~\eqref{eq:min-det-equiv}, is a standard convex optimization problem with affine constraints.  It is feasible (e.g., $r=0$, $\tau=0$ is feasible), and bounded from below.  Therefore, Slater's condition implies that strong duality holds and that the maximum is attained in the dual problem.
\end{proof}

We now explore properties of the maximizer for the dual problem.

\begin{lemma}\label{lem:u-star-beta}
If $(u^*,t^*)\in\mathbb{R}^E_{\ge 0}\times\mathbb{R}$ maximizes $\phi$, then $t^*=(|V|-1)u^*(\Gamma)$ and
\begin{equation*}
\sum_{\gamma\in\Gamma}\mathcal{N}(\gamma,e)u^*(\gamma) = t^*\beta(e)\quad\text{for all }e\in E.
\end{equation*}
\end{lemma}
\begin{proof}
The second part of the conclusion can be seen from the Lagrangian.  If the equality fails for any edge $e$, then, by sending $r(e)\to\pm\infty$, one can see that $\phi(u^*,t^*)=-\infty$, contradicting optimality.  Once this is established, summing over all edges provides the value of $t^*$.
\end{proof}

In particular, Lemma \ref{lem:u-star-beta} implies that $u^*\in\mathcal{U}_\beta$, where $\mathcal{U}_\beta$ is the convex cone defined as
\begin{equation}\label{eq:cone-ubeta}
\mathcal{U}_\beta := \left\{u\in\mathbb{R}^\Gamma_{\ge 0} : \sum_{\gamma\in\Gamma}\mathcal{N}(\gamma,e) u(\gamma) = (|V|-1)u(\Gamma)\beta(e)\text{ for all }e\in E \right\}.
\end{equation}
The connection between $u$ and $\beta$ in the definition of $\mathcal{U}_\beta$ provides a convenient formula that we shall use repeatedly in what follows.
\begin{lemma}\label{lem:u-beta}
Let $\beta\in\cP_+(E)$ and let $u\in\mathcal{U}_\beta$.  Then, for any $E'\subseteq E$,
\begin{equation*}
u(\Gamma)\beta(E') = \frac{1}{|V|-1}\sum_{\gamma\in\Gamma}|\gamma\cap E'|u(\gamma)
\end{equation*}
\end{lemma}
\begin{proof}
This follows by summing the equality in the definition of $\mathcal{U}_\beta$ over $e\in E'$.
\end{proof}

The set $\mathcal{U}_\beta$ always contains the origin, $u=0$, in which case we say it is trivial.  Next, we show that whether $0$ is the only point in $\cU_\be$ depends on whether the graph is $\beta$-homogeneous.

\begin{lemma}\label{lem:homogeneous-char-U}
The set $\mathcal{U}_\beta$ is nontrivial if and only if $G$ is $\beta$-homogeneous.
\end{lemma}
\begin{proof}
If $G$ is $\beta$-homogeneous, then any pmf $\mu$ satisfying Definition~\ref{def:homogeneous} is in $\mathcal{U}_\beta\setminus\{0\}$.  Conversely, if there exists a nonzero $u\in\mathcal{U}_\beta$, then, by rescaling, we can find $\mu\in\mathcal{U}_\beta\cap\mathcal{P}(\Gamma)$.  By the definition of $\cU_\be$, the corresponding edge usage probability vector $\eta_\mu$ together with $\beta$ satisfies~\eqref{eq:beta-eta}, showing that $G$ is $\beta$-homogeneous.
\end{proof}

The cone $\mathcal{U}_\beta$ allows us to express the Lagrangian dual objective function more succinctly.

\begin{lemma}\label{lem:dual-rewritten}
The dual problem to~\eqref{eq:min-det-equiv} is equal to
\begin{equation*}
\begin{split}
\underset{u\in\mathcal{U}_\beta}{\mathrm{maximize}}\quad&\tilde\phi(u) := \sum_{\gamma\in\Gamma}u(\gamma) - \sum_{\gamma\in\Gamma}u(\gamma)\log u(\gamma),
\end{split}
\end{equation*}
where we interpret $u(\gamma)\log u(\gamma)=0$ if $u(\gamma)=0$.
\end{lemma}
\begin{proof}
The fact that we can restrict $u$ to $\mathcal{U}_\beta$ is implied by Lemma~\ref{lem:u-star-beta}.  For $u\in\mathcal{U}_\beta$, we can choose a $t$ to eliminate the coefficients of $r(e)$ in the Lagrangian, so that
\begin{equation*}
\mathcal{L}(r,\tau,u,t) = \sum_{\gamma\in\Gamma}\exp\{\tau(\gamma)\} - \sum_{\gamma\in\Gamma}u(\gamma)\tau(\gamma).
\end{equation*}
In order to evaluate the dual objective function $\phi$ on $(u,t)$, we need to minimize this expression over $\tau$.  This minimization can be done one coordinate at a time.  In the directions $\gamma$ where $u(\gamma)=0$, $\tau(\gamma)$ will be sent to $-\infty$.  For $\gamma$ where $u$ is positive, $\tau(\gamma)$ will attain the minimum at $\tau(\gamma)=\log u(\gamma)$, yielding the lemma.
\end{proof}

The support of $u\in\mathcal{U}_\beta$ plays an important role in the interpretation of the dual, suggesting the following definition.

\begin{definition}\label{def:be-fair-trees}
Given $\beta\in\cP_+(E)$, the set of \emph{$\beta$-fair trees} is defined as the union
\begin{equation*}
\Gamma_{\beta} := \bigcup_{u\in \mathcal{U}_\beta}\supp u.
\end{equation*}
In words, $\gamma$ is in $\Gamma_\beta$ if and only if there is some $u\in\mathcal{U}_\beta$ such that $u(\gamma)>0$.
\end{definition}

We now show that optimizers of $\phi$ have supports as large as possible.

\begin{lemma}
If $u^*\ne 0$ maximizes $\tilde\phi$, then $\supp u^*=\Gamma_\beta$.  In other words, an optimal choice of $u^*$ must be positive on all $\beta$-fair trees.
\end{lemma}
\begin{proof}
By contradiction, assume that there exists $u\in\mathcal{U}_\beta$ such that $\Gamma' := \supp u\setminus\supp u^* \ne \emptyset$.  For $\epsilon > 0$, define
$u_\epsilon = u^* + \epsilon u$.  Then $u_\epsilon\in\mathcal{U}_\beta$.  For any $\gamma\in\supp u^*$, we have
\begin{equation*}
u_\epsilon(\gamma)\log u_\epsilon(\gamma) = u^*(\gamma)\log u^*(\gamma) + O(\epsilon).
\end{equation*}
If $\gamma\in \Gamma'$, however, we have
\begin{equation*}
u_\epsilon(\gamma)\log u_\epsilon(\gamma) = u(\gamma)\epsilon\log\epsilon + O(\epsilon).
\end{equation*}
Substituting shows that
\begin{equation*}
\tilde\phi(u_\epsilon) = \tilde\phi(u^*) - (\epsilon\log\epsilon)\sum_{\gamma\in\Gamma'}u(\gamma) + O(\epsilon).
\end{equation*}
As $\epsilon\to 0$, this contradicts the optimality of $u^*$.
\end{proof}

The dual problem has an interesting interpretation in terms of entropy.  For a pmf $\mu\in\mathcal{P}(\Gamma)$, define the Shannon entropy (or information theoretic entropy) $H(\mu)$ of $\mu$ as in (\ref{eq:shannon-entropy}).
Assuming the weights $\beta$ are chosen so that $\mathcal{U}_\beta$ is nontrivial, we can define the {\it $\be$-induced maximum entropy problem} as
\begin{equation}\label{eq:max-entropy}
\underset{\mu\in\mathcal{P}(\Gamma)\cap\mathcal{U}_\beta}{\text{maximize}}\quad H(\mu).
\end{equation}
In particular, we get the following interpretation of the minimum determinant problem (\ref{eq:min-det}).
\begin{lemma}\label{lem:entropy-duality}
If $\mathcal{U}_\beta$ is nontrivial, then  the infimum of the determinant in~\eqref{eq:min-det}  is equal to $\exp(H(\mu^*))$, where $\mu^*$ is the unique maximum-entropy pmf in $\mathcal{U}_\beta$. Otherwise, if $\cU_\be$ is trivial, the infimum is zero.  
\end{lemma}
\begin{proof}
Any nonzero $u\in\mathcal{U}_\beta$ can be decomposed into a positive real number $\alpha$ multiplied by a pmf $\mu\in\mathcal{P}(\Gamma)\cap\mathcal{U}_\beta$.  For such a $u$,
\begin{equation*}
\tilde\phi(u) = \tilde\phi(\alpha\mu) = \alpha + \alpha H(\mu) - \alpha\log\alpha.
\end{equation*}
If we maximize in $\alpha$ first, we see that $\alpha = \exp(H(\mu))$, so that
\begin{equation*}
\sup_{\alpha>0,\mu\in\mathcal{P}(\Gamma)\cap\mathcal{U}_\beta}\tilde\phi(\alpha\mu) = 
\sup_{\mu\in\mathcal{P}(\Gamma)\cap\mathcal{U}_\beta}\tilde\phi(\mu\exp(H(\mu)))
= \sup_{\mu\in\mathcal{P}(\Gamma)\cap\mathcal{U}_\beta}\exp(H(\mu)) > 0.
\end{equation*}
So, if $\mathcal{U}_\beta$ is nontrivial, then maximizing $\tilde\phi$ is equivalent to the maximum entropy problem.  The fact that $\mu^*$ is unique is a consequence of the fact that, as with $u^*$, $\supp\mu^*=\Gamma_{\beta}$.  Strict convexity then shows that there is only one entropy maximizing pmf in $\mathcal{U}_\beta$.

On the other hand, if $\mathcal{U}_\beta$ is trivial, then the maximum of $\tilde\phi$ is attained at $u^*=0$.  And, by Lemma~\ref{lem:slater-applied}, the infimum of the determinant is $\tilde\phi(0)=0$.
\end{proof}
We end this section by showing that any graph $G$ can be reweighed so as to become $\be$-homogeneous.
\begin{theorem}\label{thm:good-beta}
Let $G=(V,E)$ be a graph and let $\Ga$ be the family of spanning trees of $G$. Assume that $\mu\in\cP(\Ga)$ has the property that for every $e\in E$, there is $\ga\in\Ga$, with $\mu(\ga)>0$, so that $e\in\ga$.
Then, there is $\beta\in\cP_+(E)$ so that $\mu\in \mathcal{U}_\beta$. 

In particular, any graph $G$ is $\beta$-homogeneous, if we let
$\beta(e)$ be proportional to $\cReff(e)$,
where $\cReff(e)$ is the effective resistance of the edge $e$.
\end{theorem}
\begin{proof}
We want to show that if $\mu$ is as above, then $\mu$ is in $\cU_\be$ as defined in (\ref{eq:cone-ubeta}), for some $\be\in\cP_+(E)$. Note that $\mu\in\R_{\ge 0}^\Ga$ and $\mu(\Ga)=1$. Also, 
\[\sum_{\ga\in\Ga}\cN(\ga,e)\mu(\ga)=\eta_\mu(e)=\bP_\mu\left(e\in\underline{\ga}\right).\]
Next, the property that every $e\in E$ belongs to some tree $\ga\in\Ga$ with $\mu(\ga)>0$, implies that $\eta_\mu(e)>0$, for every $e\in E$.

Finally, recall that for any $\mu\in\cP(\Ga),$ we have
\begin{equation}\label{eq:edge-probabilities-sum}
\sum_{e\in E} \eta_\mu(e)= |V|-1.
\end{equation}
Therefore, we can set 
\[
\beta(e):=\frac{\eta_\mu(e)}{|V|-1}\qquad\forall e\in E.
\]
And, this defines a pmf $\be\in\cP_+(E)$ so that $\mu\in\cU_\be$.

In particular, suppose $\mu_0\in\cP(\Ga)$ is the uniform distribution. Then, $\mu_0$ is supported on every tree in $\Ga$. Moreover, any edge $e\in E$ belongs to at least one tree in $\Ga$, as can be seen by performing a greedy algorithm that adds edges as long as they don't form cycles. This implies that $\mu_0$ has the property that $\eta_{\mu_0}(e)>0$ for every $e\in E$. Finally, Kirchhoff's theorem (\ref{eq:peredgeeffres}) applied to the unweighted case, shows that
$\eta_{\mu_0}(e)=\cReff(e)$, for all $e\in E$.

\end{proof}

\section{Minimizing sequences}\label{sec:minimizing-sequences}

Now we turn our attention to minimizing sequences for the determinant.  If $G$ is not $\beta$-homogeneous, we know by Lemma \ref{lem:homogeneous-char-U} and Lemma~\ref{lem:entropy-duality} that we can make the determinant arbitrarily close to zero.  This is essentially a signal that the $\beta$ weights are not appropriate for the structure of the graph.  If $\beta$ is chosen appropriately, though, we can study the limiting behavior of minimizing sequences.  To do this, we shall use the Kullback--Leibler divergence.  We recall the relevant definitions for probability vectors (i.e., discrete distributions on finite sets).

\begin{definition}
Let $p,q\in\mathbb{R}^n$ be two probability vectors.  The \emph{Kullback--Leibler divergence} between $p$ and $q$ is defined as
\begin{align*}
D_{\KL}(p\|q) & := -\sum_{i=1}^np_i\log\frac{q_i}{p_i} = - \sum_{i=1}^np_i\log q_i + \sum_{i=1}^np_i\log p_i \\
&= - \sum_{i=1}^np_i\log q_i - H(p).
\end{align*}
The first term on the right hand-side is called the \emph{cross entropy} between $p$ and $q$ and is denoted $H(p,q)$.
\end{definition}

The Kullback--Leibler divergence is a convenient measure to use with $\UST_\si$'s due to the following lemma.

\begin{lemma}\label{lem:WUST-KL}
Let $\mu\in\mathcal{P}(\Gamma)$ with associated edge usage probability vector $\eta_\mu$, and let $\mu_\sigma$ be the pmf for a $\UST_\si$ associated with weights $\sigma$.  Then
\begin{equation*}
H(\mu,\mu_\sigma) = \log\sum_{\gamma\in\Gamma}\sigma[\gamma] - \sum_{e\in E}\eta_\mu(e)\log\sigma(e)
\end{equation*}
\end{lemma}
\begin{proof}
Indeed, if we exchange the order of summation,
\begin{equation*}
\begin{split}
H(\mu,\mu_\sigma) &= -\sum_{\gamma\in\Gamma}\mu(\gamma)\log\mu_\sigma(\gamma)
= -\sum_{\gamma\in\Gamma}\mu(\gamma)\log\left(\frac{\prod\limits_{e\in E}\sigma(e)^{\mathcal{N}(\gamma,e)}}{\sum\limits_{\gamma'\in\Gamma}\sigma[\gamma']}\right)\\
&= \log\sum_{\gamma\in\Gamma}\sigma[\gamma] - \sum_{\gamma\in\Gamma}\mu(\gamma)\sum_{e\in E}\mathcal{N}(\gamma,e)\log\sigma(e) \\
&= \log\sum_{\gamma\in\Gamma}\sigma[\gamma] - \sum_{e\in E}\left(\sum_{\gamma\in\Gamma}\mu(\gamma)\mathcal{N}(\gamma,e)\right)\log\sigma(e) ,
\end{split}
\end{equation*}
which yields the lemma.
\end{proof}

\begin{lemma}\label{lem:convergence-to-mu-star}
Suppose $G$ is $\beta$-homogeneous and let $\mu^*$ be the unique, maximum entropy pmf in $\mathcal{U}_\beta$ from Lemma~\ref{lem:entropy-duality}.  Let $\{\sigma_k\}$ be a minimizing sequence for the infimum in~\eqref{eq:min-det}, and let $\mu_{\sigma_k}$ be the corresponding $\UST_\si$ pmfs.  Then $\mu_{\sigma_k}\to\mu^*$, as $k\rightarrow\infty$.
\end{lemma}

\begin{proof}
Since $\{\sigma_k\}$ is a minimizing sequence for~\eqref{eq:min-det}, it is also a minimizing sequence for $g$ in~\eqref{eq:g-sigma}.  Using Lemma~\ref{lem:WUST-KL} and the fact that $\mu^*\in\mathcal{U}_\beta$, we see that
\begin{equation*}
H(\mu^*,\mu_{\sigma_k}) = \log\sum_{\gamma\in\Gamma}\sigma_k[\gamma] - (|V|-1)\sum_{e\in E}\beta(e)\log\sigma_k(e).
\end{equation*}
Comparing to~\eqref{eq:g-sigma}, we see that
\begin{equation}\label{eq:gsigmak}
g(\sigma_k) = H(\mu^*,\mu_{\sigma_k}) = H(\mu^*) + D_{\KL}(\mu^*\|\mu_{\sigma_k}).
\end{equation}
Since $g(\sigma_k)\to H(\mu^*)$ by Lemma~\ref{lem:entropy-duality}, it follows that $D_{\KL}(\mu^*\|\mu_{\sigma_k})\to 0$.  For finite probability vectors, this implies pointwise convergence.
\end{proof}

\begin{corollary}\label{cor:minimizer-is-WUST}
The weights $\sigma\in\mathbb{R}^E_{>0}$ minimize the determinant in~\eqref{eq:min-det} if and only if $\mu_{\sigma}=\mu^*$, where $\mu^*$ the maximum entropy pmf in $\mathcal{U}_\beta$.
\end{corollary}
\begin{proof}
If $\sigma\in\mathbb{R}^E_{>0}$ minimizes the determinant in~\eqref{eq:min-det}, then the constant sequence $\si_k\equiv\si$ is a minimizing sequence, hence $\mu_\si=\mu^*$, by Lemma \ref{lem:convergence-to-mu-star}.

Conversely, suppose that $\mu_\si=\mu^*$. Then, by (\ref{eq:gsigmak}), $g(\si)=H(\mu^*)$. Hence, $\si$ is a minimizer for~\eqref{eq:min-det}.
\end{proof}

\section{Barriers to the existence of a minimizer}\label{sec:barriers}

As shown in the previous sections, there are two barriers to the existence of a minimizer for the infimum in~\eqref{eq:min-det}.  On one hand, if $G$ is not $\beta$-homogeneous, then the determinant is not bounded away from zero.  However, the determinant is positive for any choice of positive weights $\sigma$, so no minimizer can exist.

When $G$ is $\beta$-homogeneous, however, the existence of a minimizer is more subtle. 
\begin{lemma}\label{lem:defective-mu}
Assume $\be\in\cP_+(\Ga)$ and let $\Ga_\be$ be the family of $\be$-fair trees as in Definition \ref{def:be-fair-trees}. If $G$ is $\beta$-homogeneous and~\eqref{eq:min-det} does not have a minimizer, then $\Gamma_\beta\ne\Gamma$.  In other words, in this case, the maximum entropy pmf $\mu^*$ from Lemma \ref{lem:entropy-duality} must give zero mass to some trees in $\Ga$.
\end{lemma}
\begin{proof}
Without loss of generality, we may assume that $G$ is (vertex) biconnected.
Indeed, assume that $G_1,G_2,\ldots,G_p$ are the biconnected components of $G$ and let $\Gamma_i$ be the family of spanning trees on $G_i$ for $i=1,2,\ldots,p$.  The set of spanning trees $\Gamma$ on $G$ is a direct sum of the $\Gamma_i$ in the sense that
\begin{equation}\label{eq:gamma-decomposition}
\Gamma = \left\{\gamma=\bigcup_{i=1}^p\gamma_i : \gamma_i\in\Gamma_i\text{ for }i=1,2,\ldots,p\right\}.
\end{equation}
This induces a decomposition of the determinant:
\begin{equation}\label{eq:mass-decomposition}
\sum_{\gamma\in\Gamma}\sigma[\gamma] = \sum_{\gamma_1\in\Gamma_1}\cdots\sum_{\gamma_p\in\Gamma_p}\prod_{i=1}^p\sigma[\gamma_i]
= \prod_{i=1}^p\left(\sum_{\gamma\in\Gamma_i}\sigma[\gamma]\right).
\end{equation}
Moreover, in light of Remark~\ref{rem:rescale-beta}, $\beta$ can be rescaled to a probability vector on each biconnected component, resulting in minimization problems of the form~\eqref{eq:min-det}.  Thus, in minimizing the determinant, we can minimize over each biconnected component of $G$ independently. In particular, if $G$ is not biconnected, then we can replace $G$ by one of its biconnected components on which no minimizer exists. 

Without loss of generality, we may also may assume that $|V|\ge 3$, because when $|V|=2$, there is always a minimizer for~\eqref{eq:min-det}.
Indeed, in that case, $\Ga\simeq E,$ and $|V|-1=1$. So, if we consider the edge-weights  $\be$, then $\mu_\be$ is a weighted uniform pmf with $\eta_\be=\be.$ In particular, $\mu_\be\in\cU_\be\cap\cP(\Ga)$. Therefore, $G$ is always $\beta$-homogeneous. However, in this case, \eqref{eq:min-det} always has a minimizer. Namely, suppose $\si=c\be$ for some constant $c>0$. Then, the constraint in \eqref{eq:min-det} is satisfied, if $\bar{c}=\prod_{e\in E}\be(e)^{-\be(e)}.$ Write $\bar{\si}=\bar{c}\be.$ With this choice, the determinant is $\sum_{e\in E}\bar{\si}(e)=\bar{c}.$ On the other hand, $\exp(H(\mu_\be))=\bar{c}$ as well, which shows that $\bar{\si}$ is a minimizer.

Now, assume that $\si_k$ is a minimizing sequence. We claim that there is a constant $M>0$, a subsequence, which we still call $\si_k$, and a set $E_0\subset E$ such that $\emptyset\ne E_0\subsetneq E,$ so that $\lim_{k\rightarrow\infty}\si_k(e)=0,$ for all $e\in E_0,$ while $\sigma_k(e)\ge M$ for all $k$ and all $e\in E\setminus E_0$.

To see this, suppose first that there exists a constant $c>0$ such that $c^{-1}\le\sigma_k(e)\le c$ for all $e$ and $k$. Then, we could extract a converging subsequence and, by continuity, the limit of this sequence would be a minimizer.  Thus, since we are assuming that a minimizer does not exist in this case, the minimizing sequence $\si_k$ is not bounded away from zero and infinity. Furthermore, the product constraint in~\eqref{eq:min-det}, implies that for some edges $\sigma_k(e)$ must be getting arbitrarily close to zero, while for other edges  $\sigma_k(e)$ is growing arbitrarily large. 
At first, initialize the set $E_0$ to be empty. 
Next, pick an edge $e_1$ such that $\liminf_{k\rightarrow\infty}\si_k(e_1)=0$ and extract a subsequence so that, after renaming it, we have $\lim_{k\rightarrow\infty}\si_k(e_1)=0,$ and add $e_1$ to the set $E_0$.
If $\si_k$ is not bounded below uniformly on $E\setminus E_0$, we can repeat the procedure and choose $e_2\not\in E_0$, so that after extracting another subsequence  $\lim_{k\rightarrow\infty}\si_k(e_2)=0$. Then, we update the set $E_0$ and set it equal to $\{e_1,e_2\}$. This, process must terminate before  $E_0=E$, otherwise $\si_k\rightarrow 0$ uniformly on $E$, which is impossible.
\begin{claim}\label{cl:tree-norestrict}
When $G$ is biconnected and $\emptyset\ne E_0\subsetneq E$,
there  exists a pair of edges $e_+, e_-$, with $e_+\in E\setminus E_0$ and $e_-\in E_0$, and a pair of spanning trees $\ga, \tilde{\ga}\in\Gamma$ such that $\tilde\ga=(\ga-e_-)+e_+$, meaning that $\tilde{\ga}$ is obtained from $\ga$ by removing $e_-$ and adding $e_+$.  
\end{claim}
Assuming Claim \ref{cl:tree-norestrict}, we can finish the proof of Lemma \ref{lem:defective-mu}. Indeed, since $\mu_{\si_k}(\tilde\ga)\le 1$, $\si_k(e_+)\ge M$, and $\lim_{k\rightarrow\infty}\si_k(e_-)=0$, we get 
\begin{equation*}
\mu_{\sigma_k}(\gamma) = \frac{\sigma_k(e_-)}{\sigma_k(e_+)}\mu_{\sigma_k}(\tilde{\gamma}) \le \frac{\sigma_k(e_-)}{M}\longrightarrow 0,
\end{equation*}
as $k\rightarrow\infty$. However,
by Lemma~\ref{lem:convergence-to-mu-star}, $\mu_{\sigma_k}\to\mu^*$.
Therefore, $\mu^*(\ga)=0.$
\end{proof}
\begin{proof}[Proof of Claim \ref{cl:tree-norestrict}]
Let $A$ be a connected component of the subgraph induced by $E\setminus E_0$. 
First, we show that $A$ can be assumed to be a vertex-induced subgraph of $G$. Indeed, suppose that $e=\{x,y\}$ is an edge, with $x,y\in V_A$, but such that $e\not\in E_A.$  Then, $e\in E_0$ and there is a path $p\subset A$ connecting $x$ and $y$. Let $\ga'$ be a spanning tree of $G$ containing $p$. Since $e$ forms a cycle with $p$, we can swap $e$ with any edge $e'$ in $p$ and obtain a new tree $\ga''$. So, in this case, it is enough to let $\ga:=\ga''$, $\tilde{\ga}:=\ga'$, $e_+:=e'$, and $e_-:=e.$
Hence, we may assume that every component of $G\setminus E_0$ is vertex-induced.

Now, pick an edge $e_1=\{x_1,y_1\}\in \de A$, meaning that $x_1\not\in V_A$ and $y_1\in V_A$. Since $A$ is vertex-induced, $e_1\in E_0$. Also,
since $A$ is connected and edge-induced, so that $|E_A|\ge 1$, it follows that the node $y_1$ cannot be isolated in $A$. So, there exists $y_2\in V_A$ and an edge $e_2:=\{y_1,y_2\}\in E_A.$ In particular, $e_2\in E\setminus E_0.$
Moreover, $y_2\ne x_1$ and the path $p_1$ formed by $e_1$ and $e_2$ connects $x_1$ and $y_2$.
Since $|V|\ge 3$ and $G$ is biconnected, by Menger's Theorem, there is another path $p_2$ connecting $x_1$ and $y_2$ in $G$, that is vertex-independent from $p_1$, except for the end-points. Note that, although $G$ can have multiedges, collapsing every multiedge down to a single edge does not affect the properties of biconnectedness, and  then once we have two vertex-independent paths for the simple graph, we can also obtain such paths for the multigraph.

Let $C$ be the cycle obtained by forming the union $p_1\cup p_2$. Pick a tree $\ga_1\in\Ga$ that contains the broken cycle $C\setminus e_1.$
Then, we may swap $e_1$ and $e_2$, and get a tree $\ga_2$. Therefore, in this case, we set  $\ga:=\ga_2$, $\tilde{\ga}:=\ga_1$, $e_+:=e_2$, and $e_-:=e_1.$
\end{proof}

\begin{definition}\label{def:u-restriction}
Let $u\in\mathbb{R}^{\Gamma}$ with nontrivial support and let $H=(V_H,E_H)\in\S(G)$ (Definition~\ref{def:subgraphs}).  We say that $H$ has the \emph{$u$-restriction property} if every $\gamma\in\supp u$ restricts as a tree on $H$.  In other words, if $|\gamma\cap E_H|=|V_H|-1$ for every $\gamma\in\supp u$.
\end{definition}

The $u$-restriction property is important for the present discussion because of the following lemma.
\begin{lemma}\label{lem:u-restriction}
Let $G$ be biconnected and let $u\in\mathbb{R}^E_{\ge 0}\setminus\{0\}$.  If there exists $H\in\S(G)\setminus\{G\}$ which has the $u$-restriction property, then $\supp u \ne \Gamma$.
\end{lemma}
\begin{proof}
If $H\in\S(G)\setminus\{G\}$, then $\emptyset\ne E_H\subsetneq E.$ 
So, since $G$ is biconnected, Claim \ref{cl:tree-norestrict} implies the existence of edges $e_+\in E\setminus E_H$ and $e_-\in E_H$ and trees $\ga,\tilde{\ga}\in\Ga$, so that $\tilde{\ga}=(\ga-e_-)+e_+$. In particular, $|\tilde{\ga}\cap E_H|<|\ga\cap E_H|.$ And, this means that $\ga$ and $\tilde{\ga}$ can't both restrict to a tree on $H$. Since $H$ has the $u$-restriction property, this means that $\ga$ and $\tilde{\ga}$ can't both be in $\supp u.$ Hence, $\supp u \ne \Ga.$
\end{proof}

Essentially, this shows that the barrier to the existence of a minimizer for the determinant on a $\beta$-homogeneous graph is the existence of a strict subgraph of some biconnected component $G$ with the $\mu^*$-restriction property.  We shall explore this more deeply in Section~\ref{sec:homog-and-density}.

We can now complete the proof of the first main theorem of this paper.

\begin{proof}[Proof of Theorem~\ref{thm:main}]
Lemmas~\ref{lem:entropy-duality} and~\ref{lem:homogeneous-char-U} combined show that the determinant in~\eqref{eq:min-det} is bounded away from zero if and only if $G$ is $\beta$-homogeneous.  Moreover, since the determinant is positive for all choices of $\sigma$, if $G$ is not $\beta$-homogeneous, then no minimizer exists.

Now, suppose a minimizing set of weights $\sigma$ exists.  Corollary~\ref{cor:minimizer-is-WUST} shows that $\mu_\sigma\in\mathcal{P}(\Gamma)\cap\mathcal{U}_\beta$, which implies that $G$ is strictly $\beta$-homogeneous.  On the other hand, suppose $G$ is $\beta$-homogeneous, but that no minimizing set of weights $\si$ exists.  Then Lemma~\ref{lem:defective-mu} shows that $\Gamma_\beta\ne\Gamma$.  Since every $\UST_\si$ pmf $\mu_\si$ satisfies $\supp\mu_\sigma=\Gamma$, it must be the case that $\mathcal{P}(\Gamma)\cap\mathcal{U}_\beta$ contains no $\UST_\si$ pmfs and, therefore, that $G$ is not strictly $\beta$-homogeneous.
\end{proof}

\section{Homogeneity and density}\label{sec:homog-and-density}

Now we turn our attention to the connections between homogeneity and density, with the goal of proving Theorem~\ref{thm:main-dense}.  Suppose we have a graph $G$ and weights $\beta$ for which $G$ is \emph{not} $\beta$-homogeneous.  Our first task is to find a nearby $\hat\beta$ for which $G$ is $\hat\beta$-homogeneous.  We will use $\hat\beta$ to explore the $\beta$-density of $G$ and its subgraphs.  To this end, consider the following minimization problem.
\begin{equation}\label{eq:KL-beta}
\begin{split}
\underset{\hat\beta,\hat\mu}{\text{minimize}}\quad&-\sum_{e\in E}\beta(e)\log\hat\beta(e)\\
\text{subject to}\quad& \hat\beta(e) = \frac{\sum\limits_{\gamma\in\Gamma}\mathcal{N}(\gamma,e)\hat\mu(\gamma)}{|V|-1}\quad\text{for all }e\in E,\\
& \hat\mu\in\mathcal{P}(\Gamma),\qquad \hat\beta\in\mathcal{P}_+(E).
\end{split}
\end{equation}
For context, this problem is equivalent to the problem of minimizing the Kullback--Leibler divergence between the edge pmfs $\beta$ and $\hat\beta$ subject to the stated constraints.  If we remove the first constraint, the optimization problem becomes a minimization in $\hat\beta$ alone, yielding a simple lower bound via Gibbs' Inequality
\begin{equation*}
-\sum_{e\in E}\beta(e)\log\hat\beta(e) \ge -\sum_{e\in E}\beta(e)\log\beta(e).
\end{equation*}
By strict convexity, this bound is attained if and only if $\hat\beta = \beta$.

Since the objective function in~\eqref{eq:KL-beta} diverges to $+\infty$ as any $\hat\beta(e)$ approaches $0$, any choice of $(\hat\beta,\hat\mu)$ sufficiently close to the infimum lies inside a compact set and, therefore, a minimizer exists.  In fact, by the strict convexity of the objective function, the optimal $\hat\beta$ is unique.

We shall derive necessary conditions on minimizers in terms of weighted lengths of spanning trees.

\begin{definition}
Let $v\in\mathbb{R}^E_{\ge 0}$ and let $\gamma\in\Gamma$.  Then the \emph{$v$-length} of $\gamma$ is defined as
\begin{equation*}
\ell_v(\gamma) := \sum_{e\in E}\mathcal{N}(\gamma,e)v(e) = \sum_{e\in\gamma}v(e).
\end{equation*}
\end{definition}

\begin{lemma}\label{lem:necessary-KL-beta}
Suppose $G$ is not $\beta$-homogeneous, and suppose that $(\beta^*,\mu^*)$ is a minimizer for~\eqref{eq:KL-beta}.  Define $v(e) = \beta(e)/\beta^*(e)$, for all $e\in E$. Then, for all $\gamma\in\Gamma$,
\begin{equation*}
\ell_v(\gamma) \le |V|-1.
\end{equation*}
Moreover, if $\mu^*(\gamma)>0$, then equality holds.
\end{lemma}
\begin{proof}
Let $\gamma\in\Gamma$ and let $\delta_\gamma\in\mathcal{P}(\Gamma)$ be the Dirac mass on $\gamma$.  For the real parameter $t$, define
\begin{equation*}
\mu_t = (1-t)\mu^* + t\delta_\gamma.
\end{equation*}
Defining $\beta_t$ from $\mu_t$ as in~\eqref{eq:beta-eta-verified} yields
\begin{equation*}
\beta_t(e) = (1-t)\beta^*(e) + \frac{t}{|V|-1}\mathcal{N}(\gamma,e).
\end{equation*}
By construction, $\mu_t(\Gamma)=\beta_t(E)=1$ for any choice of $t$.  Therefore, as long as $\mu_t\ge 0$ and $\beta_t>0$, we have that $(\beta_t,\mu_t)$ is feasible for~\eqref{eq:KL-beta}. Since $\beta^*>0$ by assumption, $\beta_t>0$ for all $t$ in a neighborhood of $0$.

Note that
\begin{equation*}
\log\beta_t(e) = \log\beta^*(e) - t\left(1-\frac{\mathcal{N}(\gamma,e)}{(|V|-1)\beta^*(e)}\right) + O(t^2),
\end{equation*}
so
\begin{equation*}
-\sum_{e\in E}\beta(e)\log\beta_t(e) = -\sum_{e\in E}\beta(e)\log\beta^*(e)
+ \frac{t}{|V|-1}(|V|-1-\ell_v(\gamma)) + O(t^2).
\end{equation*}
If $\mu^*(\gamma)>0$, then $(\beta_t,\mu_t)$ is feasible for $t\in(-\ep,\ep)$ for some small $\ep$, implying that $\ell_v(\gamma)=|V|-1$.  If $\mu^*(\gamma)=0$, then $(\beta_t,\mu_t)$ is feasible for small positive $t$, implying the inequality.
\end{proof}

\begin{lemma}\label{lem:hat-beta-restricts}
Suppose $G$ is not $\beta$-homogeneous, and
let $\beta^*$, $\mu^*$ and $v$ be as in Lemma~\ref{lem:necessary-KL-beta}.  Define $M=\max_e v(e)$
and let $H=(V_H,E_H)$ be any connected component of the subgraph of $G$ induced by the edges where $v(e)$ attains this maximum.  Then, $H$ has the $\mu^*$-restriction property.
\end{lemma}
\begin{proof}
Note that the max cannot be attained on all of $G$, else $\be^*=\be$, which we have excluded by assuming that $G$ is not $\be$-homogeneous.

Let $\gamma\in\Gamma$ have the property that $|\gamma\cap E_H|<|V_H|-1$.  Since $H$ is connected, this means that $\ga\cap H$ is a forest with at least two components. Let $e_+=\{x,y\}$ be an edge in $H\setminus \ga$, so that $x$ and $y$ belong to distinct components of $\ga\cap H.$ Since $\ga$ is a tree, it contains a path $p$ connecting $x$ to $y$. Also, since $x$ and $y$ belong to distinct components of $\ga\cap H$, it must be the case that the path $p$ visits an edge $e_-$ in $E\setminus E_H.$

Next, we swap $e_+$ and $e_-$ to get a new spanning tree $\gamma'=(\gamma\cup\{e_+\})\setminus\{e_-\}$.  However, by Lemma \ref{lem:necessary-KL-beta} we have,
\begin{equation*}
|V|-1 \ge \ell_v(\gamma') = \ell_v(\gamma) + v(e_+) - v(e_-) > \ell_v(\gamma).
\end{equation*}
Thus, by the `Moreover' part of Lemma~\ref{lem:necessary-KL-beta}, $\mu^*(\gamma)=0$.
\end{proof}

The proof of the connection between homogeneity and density comes from a deeper understanding of the restriction property.
\begin{lemma}\label{lem:homog-density}
Let $\cU_\be$ be the cone defined in (\ref{eq:cone-ubeta}). Suppose that $G$ is $\beta$-homogeneous, so that $\cU_\be\ne\emptyset$, and let $u\in\mathcal{U}_\beta\setminus\{0\}$. Consider a subgraph $H=(V_H,E_H)\in\S(G)$.  Then, \begin{equation}\label{eq:weakly-dense}
\theta_\beta(H)\le\theta_\beta(G),
\end{equation}
with equality holding if and only if $H$ has the $u$-restriction property.
\end{lemma}
\begin{proof}
By Lemma~\ref{lem:u-beta}, and the fact that $\be(E)=1$, we have
\begin{equation*}
\beta(E_H) = \frac{\theta_\beta(G)}{u(\Gamma)}\sum_{\gamma\in\Gamma}|\gamma\cap E_H|u(\gamma) \le \theta_\beta(G)(|V_H|-1),
\end{equation*}
since $H$ is connected, proving the inequality (\ref{eq:weakly-dense}).  Moreover, equality holds if and only if $|\gamma\cap E_H|=|V_H|-1$ for every $\gamma\in\supp u$.
\end{proof}

Now we are ready to prove that $\beta$-density and $\beta$-homogeneity are equivalent, and likewise for the strict versions of each.
\begin{proof}[Proof of Lemma~\ref{lem:dense-is-homogeneous}]
Suppose that $G$ is $\beta$-homogeneous.  By Lemma~\ref{lem:homogeneous-char-U}, there exists a nonzero $u\in\mathcal{U}_\beta$.  Let $H\in\S(G)$, then  Lemma~\ref{lem:homog-density} implies that $\theta_\beta(H)\le\theta_\beta(G)$. So, by Definition \ref{def:beta-dense}, $G$ is $\beta$-dense.

Now, suppose that $G$ is not $\beta$-homogeneous.  Let $\beta^*$, $\mu^*$, $M$ and $H$ be as in Lemmas~\ref{lem:necessary-KL-beta} and~\ref{lem:hat-beta-restricts}.  
Note that $M>1$.  Otherwise, $\beta\le \beta^*$ and, since they are both pmfs on $E$, we have $\beta=\beta^*$, which would imply that $G$ is $\beta$-homogeneous. Also, note that $\theta_\beta(G) = \theta_{\beta^*}(G)=(|V|-1)^{-1},$ and, by definition of $H$,
\begin{equation*}
\theta_\beta(H)=M\theta_{\beta^*}(H).
\end{equation*}

By Lemma~\ref{lem:hat-beta-restricts}, $H$ has the $\mu^*$-restriction property. Summing the constraint in (\ref{eq:KL-beta}) over all the edges in $E_H$, we get
\begin{equation*}
\frac{\beta^*(E_H)}{|V_H|-1} = \frac{1}{(|V_H|-1)(|V|-1)}\sum_{\gamma\in\Gamma}\mu^*(\gamma)|\gamma\cap E_H| = \frac{1}{|V|-1}.
\end{equation*}
So,
\begin{equation*}
\theta_\beta(H) = M\theta_{\beta^*}(H) > \theta_{\beta^*}(H) = (|V|-1)^{-1} = \theta_\beta(G),
\end{equation*}
which shows that $G$ is not $\beta$-dense.
\end{proof}

\begin{proof}[Proof of Lemma~\ref{lem:sdense-is-shomogeneous}]
First, suppose that $G$ is  strictly $\beta$-dense, then $G$ must be biconnected. Suppose not. Then, there exists a decomposition of $G$ into two subgraphs $H_i=(V_i,E_i)\in\S(G)$, for $i=1,2$, with the property that $|V_1\cap V_2|=1$.  That is, the graphs $H_1$ and $H_2$ are edge-disjoint and meet at a single articulation vertex. Write
\begin{equation*}
\theta_\beta(G) = \frac{\beta(E)}{|V|-1} = \frac{\beta(E_1)+\beta(E_2)}{|V|-1}
= \frac{|V_1|-1}{|V|-1}\theta_\beta(H_1) + \frac{|V_2|-1}{|V|-1}\theta_\beta(H_2),
\end{equation*}
Since $|V_1|+|V_2|=|V|+1$, we have
shown that $\theta_\beta(G)$ is a convex combination of the $\theta_\beta(H_i)$.  Since $G$ is $\beta$-dense, $\theta_\beta(H_i)\le\theta_\beta(G)$, for $i=1,2$, and therefore, both subgraphs must have the same $\beta$-density as $G$, which contradicts the assumption that $G$ is strictly $\be$-dense.

Next, assume that $G$ is strictly $\beta$-dense, but not strictly $\beta$-homogeneous. Since $G$ is in particular $\be$-dense, then by Lemma \ref{lem:dense-is-homogeneous}, $G$ is $\be$-homogeneous. Since $G$ is not strictly $\be$-homogeneous, by Corollary \ref{cor:minimizer-is-WUST}, problem~\eqref{eq:min-det} does not have a minimizer. Hence,
as shown in the proof of Lemma~\ref{lem:defective-mu}, we can find a subgraph $H\in\S(G)$ (a connected component of $G\setminus E_0$) that has the $\mu^*$-restriction property, where $\mu^*$ is the entropy maximizing pmf from~\eqref{eq:max-entropy}. Lemma~\ref{lem:homog-density} implies that $\theta_\beta(H)=\theta_\beta(G)$, but this contradicts the fact that $G$ is strictly $\beta$-dense.

For the other direction, suppose that $G$ is strictly $\beta$-homogeneous and biconnected.  Then there exists a $\UST_\si$ pmf $\mu_\sigma\in\mathcal{U}_\beta$, with corresponding edge usage probability vector $\eta_\sigma$ satisfying $\beta = \theta_\beta(G)\eta_\sigma$.  Let $H\in\S(G)$.  Since $G$ is biconnected and since $\supp\mu_\sigma=\Gamma$, Lemma~\ref{lem:u-restriction} shows that $H$ cannot have the $\mu_\sigma$-restriction property.  So, by Lemma~\ref{lem:homog-density}, $\theta_\beta(H)<\theta_\beta(G)$, showing that $G$ is strictly $\beta$-dense.
\end{proof}

\begin{proof}[Proof of Theorem~\ref{thm:main-dense}]
By Theorem \ref{thm:main} and Lemma \ref{lem:dense-is-homogeneous}, we see that the determinant in~\eqref{eq:min-det} is bounded away from zero if and only if $G$ is $\beta$-dense. 

Now, assume that a minimizing $\si$ exists. Then, by Theorem \ref{thm:main}, $G$ is strictly homogeneous. If we decompose $G$ into its biconnected components $\{G_i\}_{i=1}^p$, and use (\ref{eq:gamma-decomposition}) and
(\ref{eq:mass-decomposition}), we see that each component is also strictly homogeneous. Therefore, by Lemma \ref{lem:sdense-is-shomogeneous}, each component is strictly $\be$-dense. Conversely, if every biconnected component is strictly $\be$-homogeneous, then, by Lemma \ref{lem:sdense-is-shomogeneous}, each component is strictly $\be$-dense. Now suppose $H\in\cS(G)$. Write $H_i=H\cap G_i$, for $i=1,\dots,p$. Then, 
\[
|V_H|-1=\sum_{i=1}^p (|V_{H_i}|-1).
\]
Therefore, as we did in the proof of Lemma~\ref{lem:sdense-is-shomogeneous}, we can write $\theta_\be(H)$ as a convex combination:
\[
\theta_\be(H) = \sum_{i=1}^P t_i \theta_\be(H_i),
\]
where $0\le t_i\le 1$ and $\sum_{i=1}^p t_i=1.$ And, this implies that $G$ itself is strictly $\be$-dense.
\end{proof}

\section{Determinant bounds and the FEU problem}\label{sec:det-bounds}

The analysis presented above provides an answer to the following problem.  Suppose the graph $G$ is given.  Under what conditions on the positive weights $b\in\mathbb{R}^E_{>0}$ can we find a bound of the form
\begin{equation}\label{eq:determinant-bound}
\detp L_\sigma \ge c\prod_{e\in E}\sigma(e)^{b(e)}\quad\text{for all }\sigma\in\mathbb{R}^E_{>0}
\end{equation}
for some $c>0$?

\begin{theorem}\label{thm:determinant-bound}
Given $c>0$ and $b\in\R_{>0}^E$, an inequality of the type~\eqref{eq:determinant-bound} holds for $G$, if and only if $b(E)=|V|-1$ and $G$ is $\beta$-homogeneous, with $\beta:=(|V|-1)^{-1}b$.  The largest choice of constant $c$ in this case is
\begin{equation}\label{eq:optimal-c}
c^* = |V|\exp(H(\mu^*))
\end{equation}
where $\mu^*$ is the entropy maximizing pmf in $\mathcal{U}_\beta$.  Moreover, the bound is attained for some $\sigma\in\R_{>0}^E$ if and only if $G$ is strictly $\beta$-homogeneous.
\end{theorem}
\begin{proof}
The goal is to bound the function
\begin{equation*}
F(\sigma) := \detp L_\sigma\prod_{e\in E}\sigma(e)^{-b(e)}
\end{equation*}
away from zero.
Taking a logarithm and using Theorem~\ref{thm:matrix-tree}, we see that this is equivalent to finding a positive $c$ such that, for every $\sigma\in\R_{>0}^E$,
\begin{equation}\label{eq:bound-log}
\log\frac{c}{|V|} \le \log\sum_{\gamma\in\Gamma}\sigma[\gamma] - \sum_{e\in E}b(e)\log\sigma(e).
\end{equation}
Note that the first term on the right hand-side scales as follow, for $t>0$,
\[
\log\sum_{\gamma\in\Gamma}(t\sigma)[\gamma]=\log\sum_{\gamma\in\Gamma}\sigma[\gamma]+(|V|-1)\log t,
\]
while the second term scales as
\[
\sum_{e\in E}b(e)\log(t\sigma)(e)=\sum_{e\in E}b(e)\log\sigma(e)+b(E)\log t.
\]
Therefore, in order for a lower bound as in (\ref{eq:bound-log}) to hold, it is necessary that $b(E)=|V|-1$.
Hence, setting $\beta:=(|V|-1)^{-1}b$, we get that $\beta\in\cP_+(E)$. Next, we can rewrite (\ref{eq:bound-log}) as
\[
\log\frac{c}{|V|} \le \log\sum_{\gamma\in\Gamma}\sigma[\gamma] - (|V|-1)\sum_{e\in E}\be(e)\log\sigma(e).
\]
In particular, the right hand-side can now be identified with the function $g$ defined in (\ref{eq:g-sigma}).
Thus, by Lemma~\ref{lem:homogeneous-char-U} and Lemma~\ref{lem:entropy-duality}, for (\ref{eq:bound-log}) to hold, it is necessary and sufficient that $G$ be $\beta$-homogeneous.  This is equivalent to saying that $\beta$ is a strictly positive edge-usage probability vector for some pmf $\mu\in\mathcal{P}(\Gamma)$.

In this case, the infimum of $g(\sigma)$ over all weights is equal to $H(\mu^*)$, yielding the optimal value $c^*$ in (\ref{eq:optimal-c}).  By Corollary~\ref{cor:minimizer-is-WUST}, a minimizing choice of weights $\sigma$ exists if and only if $\mu^*=\mu_\sigma$, which is possible if and only if $G$ is strictly $\beta$-homogeneous.
\end{proof}

In general, then, there are many choices of $b$ which support an inequality of the form in~\eqref{eq:determinant-bound}.  Any valid choice of $b$ has the average value of
\begin{equation*}
\frac{1}{|E|}\sum_{e\in E}b(e) = \frac{|V|-1}{|E|}
\end{equation*}
across all edges.  It is interesting to ask how close we can make $b$ to the constant vector.  For instance, we might ask for the value of $b$ supporting~\eqref{eq:determinant-bound} that has smallest variance.  This gives rise to the fairest edge usage (FEU) problem analyzed in~\cite{achpcst:disc-math2021}.

In the language of that paper, a graph was called homogeneous if the minimum variance is zero (i.e., if $\beta$ can be made constant). By Theorem~\ref{thm:determinant-bound}, the homogeneous graphs are exactly those graphs that support an inequality of the form
\begin{equation*}
\detp L_\sigma \ge c\left(\prod_{e\in E}\sigma(e)\right)^{\frac{|V|-1}{|E|}},
\end{equation*}
for some constant $c>0$,
and the strictly homogeneous graphs are those for which the bound can be attained for some $\sigma>0$ and $c>0$.

\section{Applications to spanning tree modulus}\label{sec:connections-to-sptmod}

The theory of spanning tree modulus was first proposed in \cite{achpcst:disc-math2021} and we refer to that paper for all the necessary background. The idea is to think of the spanning trees $\Ga$ as a family of objects and take its $2$-modulus, the way one does for the family $\Ga(a,b)$ of all paths from $a$ to $b$. In the case of $\Ga(a,b)$, the $2$-modulus recovers the effective conductance from $a$ to $b$, in the sense of electrical networks. The dual problem of $2$-modulus has an interesting probabilistic interpretation we call the Minimum Expected Overlap ($\MEO$) problem. Namely, given a pmf $\mu\in\cP(\Ga)$, the expected overlap of two independent random trees with law $\mu$ can be computed using the usage matrix in (\ref{eq:usage-matrix}) as follow:
\[
\bE_\mu\left|\underline{\ga}\cap\underline{\ga'}\right|=\sum_{\ga\in\Ga}\sum_{\ga'\in\Ga}\mu(\ga)\mu(\ga')|\ga\cap\ga'|=\mu^T\cN\cN^T\mu.
\]
The goal of $\MEO$ is to minimize this expected overlap over all $\mu\in\cP(\Ga).$ Optimal pmf's $\mu^*$ always exist but are not necessarily unique, while the corresponding edge probabilities $\eta^*(e)=\bP_{\mu^*}\left(e\in\underline{\ga}\right)$ are uniquely defined and are related to another modulus problem for the family of feasible partitions.

The optimal density $\eta^*$ turned out to be very important. For instance, the set where $\eta^*$ attains its minimum value gives rise to the notion of homogeneous core, which is related to the combinatorial problem of finding a subgraph of maximum denseness \cite[Theorem 5.9]{achpcst:disc-math2021}, and is the main step in the process of deflation \cite[Theorem 5.7]{achpcst:disc-math2021}.

The results in this paper lead to a more precise version of the maximum denseness problem, which we call the {\it strict denseness} problem. In this section we will assume that $\be$ is constant, and refer to graphs that are {\it dense} or {\it strictly dense}, meaning with respect to $\be$ constant.
Recall the set of all subgraphs $\cS(G)$ from Definition \ref{def:subgraphs}.
Then, we say that $H\in\cS(G)$ solves the strict denseness problem, if $\theta(H)=\max_{H'\in\cS(G)}\theta(H')$ and any  proper subgraph $H''$ of $H$ satisfies $\theta(H'')<\theta(H)$. In other words, $H$ solves the maximum denseness problem and is a non-trivial connected subgraph that is itself strictly dense.
When this happens, we say that $H$ is a {\it minimal core} for $G$. Also, we define the {\it shrunk graph} $G/H$ obtained by identifying every vertex in $H$ to a single vertex $v_H$, and removing all resulting self-loops, but keeping all resulting multi-edges.
The following deflation result is similar to \cite[Theorem 5.7]{achpcst:disc-math2021} but uses minimal cores instead.
\begin{theorem}\label{thm:mincore}
Let $G=(V,E)$ be a biconnected graph, possibly  with multi-edges, but no selfloops. Then, either 
\bi
\item[(i)] $G$ is strictly dense, or 
\item[(ii)] $G$ admits a  minimal core $H_0\subsetneq G$. Moreover, in this case, 
\bi
\item[(a)] $H_0$ has the {\rm fair restriction property} with respect to $G$, meaning that for every fair tree $\ga\in\Ga_G^f$, the restriction $\ga\cap H_0$ is a spanning tree in $H_0$;
\item[(b)] for any fair tree $\ga$ of $G$, the projection of $\ga\setminus H_0$ onto the shrunk graph $G/H_0$ is a fair spanning tree of $G/H_0$; 
\item[(c)] fair trees of the core $H_0$ can be coupled with fair trees of the shrunk graph $G/H_0$ to produce fair trees of $G$.
\ei
\ei
\end{theorem}
Note that, if $G$ satisfies Theorem \ref{thm:mincore} (ii), then $G$ may or may not be homogeneous.
However, if $G$ is homogeneous and not strictly dense, then it must admit a proper subgraph with the same density. This prompted the following observation by Jason Clemens, see \cite{clemens-phdthesis}.
\begin{corollary}\label{cor:relprime}
If $G$, as in Theorem \ref{thm:mincore}, is homogeneous and the numbers $|E|$ and $|V|-1$ are relatively prime, then $G$ must be strictly dense.
\end{corollary}
\begin{proof}[Proof of Corollary \ref{cor:relprime}]
If not, by Theorem \ref{thm:mincore} (ii), $G$ admits  a proper subgraph $H_0$ with the same density. Hence, the fraction $|E(G)|/(|V(G)|-1)$ can be non-trivially simplified to $|E(H_0)|/(|V(H_0)|-1)$.
\end{proof}
\begin{proof}[Proof of Theorem \ref{thm:mincore}]
Assume $G$ is not strictly dense. Let $H_0\in\cH(G)$ solve the strict denseness problem for $G$. Then,  $H_0$ is a proper subgraph of $G$ and $\theta(H_0)\ge \theta(G)$. Note that, since $H_0$ has the strict density property, by Lemma \ref{lem:dense-is-homogeneous}, $H_0$ is also homogeneous.

Note that $H_0$ is necessarily vertex-induced, because adding edges without changing the vertices can only increase the density $\theta(H)$. It follows from \cite[Theorem 5.9]{achpcst:disc-math2021} that $H_0$ is a homogeneous core of $G$, as in \cite[Definition 5.3]{achpcst:disc-math2021}. 
In particular, $H_0$ must have the restriction property \cite[Definition 5.1]{achpcst:disc-math2021}. 

Finally parts (b) and (c) follow from part (a), because Theorem 5.5, Lemma 5.6 and Theorem 5.7, in \cite{achpcst:disc-math2021} only depend on the fact that $H_0$ has the restriction property with respect to $G$.
\end{proof}

In particular, using \cite[Theorem 5.7]{achpcst:disc-math2021}, we can formulate a serial rule for the Minimum Expected Overlap problem ($\MEO$), for the definition see (1.2) in \cite{achpcst:disc-math2021}. The advantage of performing a deflation process that repeatedly identifies a minimal core and shrinks it, is that minimal cores are strictly dense, and therefore, by Theorem \ref{thm:main-dense}, they have an optimal $\UST_\si$ pmf. In other words,
if $G$ is an arbitrary multigraph, then it always admits an optimal pmf for the $\MEO$ problem, that can be constructed from a sequence of $\UST_\si$ pmfs, and the advantage of an $\UST_\si$ pmf $\mu_\si$ is that there are well-known algorithms, like Aldous-Broder or Wilson's algorithm, for sampling them.

\bibliographystyle{acm}
\bibliography{pmodulus}
\def\cprime{$'$}

\end{document}